\documentclass[a4paper,final]{amsart}

\usepackage[margin=1.3in]{geometry}
\usepackage{amsmath}%
\usepackage{amsfonts}%
\usepackage{amssymb}%
\usepackage{amsthm}
\usepackage{graphicx}
\usepackage{dsfont}
\usepackage{pdfcomment}
\usepackage{tikz}
\usepackage{tikz-3dplot}
\usepackage{subfigure}

\usepackage{thmtools}
\usepackage{thm-restate}
\usepackage{hyperref}
\usepackage{cleveref}
\usepackage{esvect}
\usepackage{enumerate}

\usepackage[parfill]{parskip}
\begingroup
\makeatletter
\@for\theoremstyle:=definition,remark,plain\do{%
	\expandafter\g@addto@macro\csname th@\theoremstyle\endcsname{%
		\addtolength\thm@preskip\parskip
	}%
}
\endgroup

\usetikzlibrary{decorations.markings}
\newtheorem{theorem}{Theorem}[section]

\newtheorem{conjecture}[theorem]{Conjecture}
\newtheorem{corollary}[theorem]{Corollary}

\newtheorem{definition}[theorem]{Definition}
\newtheorem{example}[theorem]{Example}
\allowdisplaybreaks

\newtheorem{lemma}[theorem]{Lemma}

\newtheorem{proposition}[theorem]{Proposition}
\newtheorem{remark}[theorem]{Remark}

\DeclareMathOperator{\D}{D}
\DeclareMathOperator{\Ker}{Ker}
\newcommand{\field}[1]{\mathbb{#1}}

\DeclareMathOperator{\Imaginary}{Im}

\renewcommand{\Im}{\Imaginary}

\title{Infinitesimal conformal deformations of triangulated surfaces in space}
\author{Wai Yeung Lam        \and
	Ulrich Pinkall
}

\address{W. Y. Lam \\
	Department of Mathematics, Brown University, Providence, RI 02912}
\address{
	U. Pinkall \\
	Institut f\"ur Mathematik, Technische Universit\"{a}t Berlin, Stra{\ss}e des 17.\ Juni 136, 10623 Berlin, Germany 
}

	\email{lam@math.brown.edu, pinkall@math.tu-berlin.de}

\begin{document}
	
\begin{abstract}
We study infinitesimal conformal deformations of a triangulated surface in Euclidean space and investigate the change in its extrinsic geometry. A deformation of vertices is conformal if it preserves length cross-ratios. On one hand, conformal deformations generalize deformations preserving edge lengths. On the other hand, there is a one-to-one correspondence between infinitesimal conformal deformations in space and infinitesimal isometric deformations of the stereographic image on the sphere. The space of infinitesimal conformal deformations can be parametrized in terms of the change in dihedral angles, which is closely related to the Schl\"{a}fli formula. 
\end{abstract}

\thanks{This research was supported by the DFG Collaborative Research Centre SFB/TRR 109 \emph{Discretization in Geometry and Dynamics}.}

\maketitle

\section{Introduction}
\label{intro}

Realizing a triangulated surface in Euclidean space with prescribed edge lengths is a classical problem in rigidity theory \cite{Whiteley2014}. Fixing a combinatorial structure and a discrete metric, one is interested in determining the existence and uniqueness of the realization, which is analogous to the problem of isometric immersions of surfaces in differential geometry. It stimulates various directions of research, such as infinitesimal rigidity. A triangulated surface in space is infinitesimally rigid if all its first-order isometric deformations are induced by Euclidean motions. Dehn's rigidity theorem \cite{Dehn1916} states that all convex polyhedra are infinitesimally rigid. Gluck \cite{Gluck1975} further showed that generic triangulated spheres are infinitesimally rigid.

Rather than insisting on edges lengths, we are interested in infinitesimal deformations preserving conformal structures.

The concept of discrete conformality arose from William Thurston's idea to approximate conformal maps by circle packings in the plane \cite{Stephenson2005}. Rodin and Sullivan \cite{Rodin1987} proved the convergence of the analogue of Riemann maps for circle packings. There are further extensions where circles intersect each other \cite[Chap. 13]{Thurston1982}, such as Schramm's orthogonal circle patterns \cite{Schramm1997}. The intersection angles of circles yield a discrete notion of conformal structure which is well defined in M\"{o}bius geometry since M\"{o}bius transformations map circles to circles and preserve their intersection angles. 

For a given triangulated surface in space, one might be tempted to define its conformal structure in terms of the intersection angles of the circumscribed circles. However, deformations preserving edge lengths generally do not preserve the intersection angles because of the change in dihedral angles. This phenomenon is dissatisfying since it is inconsistent with the smooth theory, where isometric deformations are special cases of conformal deformations. Instead, Bobenko and Schr\"{o}der \cite{Bobenko2005a} related the intersection angles to the Willmore energy . On the other hand, in an effort to remedy the problem, one might measure the intersection angles after flattening neighboring triangles into the plane so as to remove the dependence on dihedral angles \cite{Kharevych:2006}. Nevertheless, the angles measured in this way are not invariant under M\"{o}bius transformations.

We consider another notion of discrete conformality that is invariant under M\"{o}bius transformations -- length cross ratios. It was proposed first in terms of vertex scaling by Luo and later written in the form of cross ratios \cite{Williams1984,Luo2004,Springborn2008}. 
\begin{definition}\label{def:confequi}
	The \textbf{length cross ratio} $\mbox{lcr}: E_{int} \to \field{R}$ of a triangulated surface equipped with a discrete metric $\ell: E \to \field{R}_{>0}$ is
	\[
	\mbox{lcr}_{ij} := \frac{\ell_{il}\ell_{jk}}{\ell_{lj}\ell_{ki}} =\mbox{lcr}_{ji}
	\]
	where $\{ijk\},\{ilj\}$ are the left and the right triangles of the edge oriented from vertex $i$ to $j$. Two discrete metrics are \textbf{conformally equivalent} if their length cross ratios are identical. Equivalently, two metrics $\ell,\tilde{\ell}$ are conformally equivalent if they differ by \textbf{vertex scaling}, i.e. there exists a \textbf{scale factor} $u:V \to \mathbb{R}$ such that
	\[
	\tilde{\ell}_{ij} = e^{\frac{u_i+u_j}{2}} \ell_{ij}.
	\] 
\end{definition}

Vertex scaling mimics the smooth theory that two Riemannian metrics $g,\tilde{g}$ are conformally equivalent if there exists $u$ such that $\tilde{g}=e^{u} g$.

Length cross ratio theory is a counterpart of circle patterns in the plane. For a triangle mesh in the plane, a complex number is associated to each interior edge by taking the cross-ratio of the four vertices of the adjacent triangles. The magnitude of the cross-ratio is the length cross-ratio. The other half of the cross-ratio, namely the argument, yields the intersection angle of the circumscribed circles. In the plane, the infinitesimal deformations of the two types are simply related by a $\pi/2$-rotation \cite{Lam2016a}.

Previous study of length cross ratios is restricted to intrinsic geometry, i.e. only referring to discrete metrics but not realizations in $\mathbb{R}^3$. Luo \cite{Luo2004} introduced vertex scaling to study combinatorial Yamabe flow. Bobenko, Pinkall and Springborn \cite{Bobenko2010} further established its relation to ideal hyperbolic polyhedra.  

Conformal deformations in space are interesting not only from the theoretical point of view but also for applications. Numerical approximations for conformal deformations of smooth surfaces have been obtained by directly discretizing equations from the smooth theory. Gu and Yau \cite{Gu2003} studied conformal parametrizations of triangulated surfaces.  Conformal deformations with respect to extrinsic geometry were also considered numerically \cite{Crane2011}. 

In this paper, we focus on infinitesimal conformal deformations of triangulated surfaces in space with respect to length cross ratio theory. It turns out in the context of the extrinsic geometry, the length cross ratio is a better notion of discrete conformality. One evidence is that the theory of length cross ratios is compatible with isometric deformations as expected in the smooth theory.

\begin{proposition}\label{thm:isoconf}
	Given a non-degenerate realization $f:V \to \field{R}^n$ of a triangulated surface, the space of infinitesimal conformal deformations of $f$ in $\field{R}^n$ is isomorphic to the space of infinitesimal isometric deformations of $\iota \circ \Phi \circ f$ in $\field{R}^{n+1}$. Here  $\Phi: \field{R}^n \to S^n $ is the stereographic projection and $\iota: S^{n} \to \mathbb{R}^{n+1}$ is the inclusion map.
\end{proposition}

We then study infinitesimal conformal deformations in $\field{R}^3$. Infinitesimal conformal deformations in $\mathbb{R}^3$ can be parametrized by scale factor $u$. These parameters are \emph{intrinsic} since they describe the change in the discrete metrics.

\begin{proposition}\label{prop:givenu}
	Given an infinitesimally rigid triangulated sphere in $\mathbb{R}^3$ and a function $u:V \to \mathbb{R}$, there exists an infinitesimal conformal deformation unique up to Euclidean motions with scale factor $u$.
\end{proposition}

Infinitesimal conformal deformations can also be parametrized by dihedral angles between face normals, analogous to the mean curvature in the smooth theory. These parameters are \emph{extrinsic}. Under an infinitesimal conformal deformation, we consider the \textbf{change in mean curvature half-density} $\rho:V_{int} \to  \field{R}$ defined on interior vertices,
\[
\rho_i = \frac{1}{2}\sum_j \dot{\alpha}_{ij}|f_j-f_i|.
\]
where $\dot{\alpha}$ denotes the change in dihedral angles. Together with the discrete Dirac operator $\D$ (Definition \ref{def:dirac}), we prove the following:
\begin{theorem}[Conformal deformations with prescribed change in mean curvature half density] \label{thm:givenrho}
	Suppose a closed triangulated sphere does not possess any non-trivial infinitesimal conformal deformation with vanishing change in mean curvature half density, i.e. we have $\dim \Ker D =4$. Then, given any $\rho:V \to \field{R}$ with $\sum_i \rho_i=0$, there exists an infinitesimal conformal deformation with $\rho$ as the change in mean curvature half density. The deformation is unique up to a similarity transformation.
\end{theorem}
Here the condition $\sum_i \rho_i=0$ is a consequence of the Schl\"{a}fli formula (Proposition \ref{prop:schafli}) and the assumption $\dim \Ker D =4$ is analogous to infinitesimal rigidity in Proposition \ref{prop:givenu}.

It has been shown \cite{Lam2016a} that infinitesimal conformal deformations of triangulated disks in the plane are closely related to discrete complex analysis. Each infinitesimal conformal deformation corresponds to a discrete harmonic function with respect to the cotangent Laplacian (Corollary \ref{cor:planarconf}). The study of the planar case has led to a unified theory of discrete minimal surfaces \cite{Lam2016b,Lam2017}.

Combining with previous results \cite{Lam2016}, we obtain an analogy between isometric deformations and conformal deformations (Table \ref{tab:1}). Euclidean motions induce trivial isometric deformations. The class of infinitesimally flexible surfaces is preserved under projective transformations. In contrast, M\"{o}bius transformations induce trivial conformal deformations. Earlier, we have studied ``conformally flexible" surfaces and called them \emph{isothermic triangulated surfaces} \cite{Lam2016}. More precisely, a closed triangulated surface in space is isothermic if its space of infinitesimal conformal deformations has dimension strictly larger than $V+6-6g$. The class of isothermic triangulated surfaces is preserved under M\"{o}bius transformations.

\begin{table}[h]
	\caption{Comparison between infinitesimal isometric and conformal deformations} 	\label{tab:1}
	\begin{tabular}{ c| c| c } 
		Types of  & \textbf{Isometric} & \textbf{Conformal}  \\
		infinitesimal deformations: & & \\ \hline 
		Constraints: & Edge lengths & Length cross-ratios \\  
		Trivial deformations: & Euclidean transformations & M\"{o}bius transformations \\
		Singularity: & Infinitesimally flexible surfaces & Isothermic surfaces \\
		Bijective under: & Projective transformations & M\"{o}bius transformations
	\end{tabular}
\end{table}

Our approach is motivated by the method of quaternionic analysis in the smooth theory, which relates conformal deformations to mean curvature \cite{Pedit1998}. In the smooth theory, a pair of non-congruent surfaces is a Bonnet pair if they are isometric with identical mean curvature. Using quaternionic analysis, there is an elegant way to obtain Bonnet pairs from an isothermic surface \cite{Kamberov1998}.

In Section \ref{sec:3} we review the theory of length cross-ratios. We then prove Proposition \ref{thm:isoconf} in Section \ref{sec:isoconfomal} and Proposition \ref{prop:givenu} in Section \ref{sec:3}. In section \ref{sec:5}, we develop the main theorem \ref{thm:uZ} that describes infinitesimal conformal deformations. An immediate corollary is the angular velocity equation in Section \ref{sec:6}. In order to relate infinitesimal conformal deformations to the change in mean curvature half-density as in Theorem \ref{thm:givenrho}, a discrete Dirac operator is introduced in Section \ref{sec:8}. Examples are given in Section \ref{sec:9}. In Section \ref{sec:infinhighdisc}, we extend the results to surfaces of high genus. Finally, the relation to isothermic surfaces and open problems are discussed in Section \ref{sec:10}.

\section{Notation} \label{sec:2}
\begin{definition}
	A triangulated surface $M=(V,E,F)$ is a finite simplicial complex whose underlying topological space is an oriented 2-manifold with or without boundary. The set of vertices (0-cells), edges (1-cells) and triangles (2-cells) are denoted as $V$, $E$ and $F$. 
\end{definition}

\begin{definition}
	A realization of a triangulated surface in $\field{R}^n$ is a map $f:V \to \field{R}^n$ which can be extended linearly to each face. We say $f$ is \emph{non-degenerate} if every face of $f$ spans an affine 2-plane. In particular it implies $f_i \neq f_j$ for every edge $\{ij\}\in E$.
\end{definition}

\begin{definition}
	A \emph{discrete metric} on a triangulated surface is a function $\ell: E \to \field{R}_{>0}$ satisfying the triangle inequality. Two discrete metrics on a triangulated surface are \emph{isometric} if they are identical.
\end{definition}

We are interested in discrete metrics induced from realizations into $\field{R}^n$.

\begin{definition}
	Every  realization $f:V \to \field{R}^n$ of a triangulated surface induces a discrete metric $\ell:E \to \field{R}_{>0}$
	\[
	\ell_{ij} = |f_j - f_i|  \quad \forall\, \{ij\} \in E.
	\]
	where $|\cdot |$ is the Euclidean norm. Two  realizations are conformally equivalent if their induced edge lengths are conformally equivalent.
\end{definition}

Without further notice all triangulated surfaces under consideration are assumed to be oriented and the realizations are non-degenerate. Each triangular face is represented by its vertices in an order that is compatible with the orientation of the surface. For example, in Fig. \ref{fig:leftright}, the left face is represented as $\{ijk\}=\{jki\}=\{kij\}$. A vector field $Z:F \to \mathbb{R}^3$ on faces obeys the same rule: $Z_{ijk}=Z_{jki}=Z_{kij}$. The opposite orientation on faces will not be considered in this paper.

An \emph{interior edge} is a common edge of two faces. We denote $V_{int}$ and $E_{int}$ the set of interior vertices and the set of interior edges respectively. We write $e_{ij}$ as the oriented edge from the vertex $i$ to the vertex $j$. Note that $e_{ij} \neq e_{ji}$. The set of oriented edges is denoted by $\vv{E}$. The set of interior oriented edges is indicated by $\vv{E}_{int}$.

We write $M^*$ as the \emph{dual mesh} of $M$ constructed as follows: A vertex of $M^*$ is associated to every face of $M$. Two vertices of $M^*$ are connected by an edge if the corresponding faces of $M$ share a common edge. Edges of $M^*$ bound a face if the corresponding edges are exactly the neighbors of some vertex of $M$. 

We make use of discrete differential forms from Discrete Exterior Calculus \cite{Desbrun2006a}. A function $\omega:\vv{E}\to \field{R}$ is called a (primal) \textbf{discrete 1-form} if
\[
\omega(e_{ij}) = -\omega(e_{ji}) \quad \forall e_{ij} \in \vv{E}.
\]
It is \textbf{closed} if for every face $\{ijk\}$	
\[
\omega(e_{ij})+\omega(e_{jk})+\omega(e_{ki})=0.
\]
It is \textbf{exact} if there exists a function $f:V \to \field{R}$ such that for $e_{ij} \in \vv{E}$
\[
\omega(e_{ij}) =  f_j -f_i =: df(e_{ij}).
\]
In particular, every exact 1-form is closed.

Similarly, we consider discrete 1-forms on the dual mesh $M^*$ and these are called \textbf{dual 1-forms} on $M$. For every oriented edge $e$, we write $e^*$ as its dual edge oriented from the right face of $e$ to its left face. A function $\eta:\vv{E}^*_{int} \to \field{R}$ defined on oriented dual edges is called a dual 1-form if
\[
\eta(e^*_{ij}) = -\eta(e^*_{ji}) \quad \forall e^*_{ij} \in \vv{E}^*_{int}.
\]
A dual 1-form $\eta$ is \emph{closed} if for ever interior vertex $i \in V_{int}$ 
\[
\sum_j \eta(e^*_{ij}) = 0.
\] 
It is exact if there exists $h:F \to \field{R}$ such that
\[
dh(e^*_{ij}):= h_{ijk}-h_{ilj} = \eta(e^*_{ij})
\] 
where $\{ijk\}$ denotes the left face of $e_{ij}$ and $\{ilj\}$ denotes the right face (Figure \ref{fig:leftright}).

\begin{figure}
	\centering
	\includegraphics[width=0.35\textwidth]{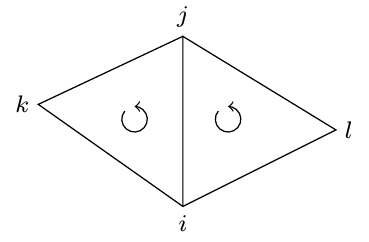}
	\caption{Two neighboring triangles contain the oriented edge $e_{ij}$}
	\label{fig:leftright}
\end{figure}

Given a realization $f:V \to \field{R}^3$ of a triangulated surface, we measure its dihedral angles $\alpha: E_{int} \to \field{R}$ between face normals. We denote $N:F \to \field{S}^2$ the \textbf{face normal} compatible with the orientation of the surface
\[
N_{ijk}:= \frac{(f_j-f_i)\times (f_k-f_i)}{|(f_j-f_i)\times (f_k-f_i)|} = N_{jki}=N_{kij}
\]
where $ijk$ is in the order consistent with the orientation. In principle $N_{ikj}=-N_{ijk}$ since $ikl$ has the opposite orientation, but this case will not appear in the paper.

The sign of the \textbf{dihedral angle} $\alpha_{ij} \in (-\pi,\pi)$ is determined by
\begin{align*}
	\sin \alpha_{ij} &= \langle N_{ijk} \times N_{ilj}, \frac{f_j-f_i}{|f_j-f_i|} \rangle, \\
	\cos  \alpha_{ij} &= \langle N_{ijk},N_{ilj} \rangle 
\end{align*}
where $\{ijk\},\{ilj\}\in F$ denote the left and the right face of $e_{ij}$.

\section{Conformal equivalence}\label{sec:3}

This section reviews the definition of the conformal equivalence of triangulated surfaces based on length cross ratios \cite{Luo2004,Springborn2008}, which possesses properties as in the smooth theory. Every immersion of a triangulated surface into Euclidean space induces a conformal structure. The conformal structure is preserved under M\"{o}bius transformations.


\begin{proposition}
	Suppose $f:V\to \field{R}^n$ is a  realization of a triangulated surface. Then for any M\"{o}bius transformation $\Phi$, the realizations $f$ and $\Phi \circ f$ are conformally equivalent.
\end{proposition}
\begin{proof}
	M\"{o}bius transformations are generated by translations, dilations and the inversion under the unit sphere. Conformal equivalence is preserved obviously under dilations and translations. Thus, it suffices to consider $\phi$ as the inversion under the unit sphere. Since
	\begin{align*}
		| \frac{f_i}{|f_i|^2}-\frac{f_j}{|f_j|^2}|^2 &= \frac{1}{|f_i|^2} + \frac{1}{|f_j|^2}- \frac{2}{|f_i|^2 |f_j|^2} \langle f_i,f_j \rangle \\
		&= \frac{1}{|f_i|^2 |f_j|^2} |f_i - f_j|^2
	\end{align*}
	where $|\cdot |$ denotes the Euclidean norm, we conclude $f$ and $\phi \circ f$ are conformally equivalent.  
\end{proof}

We consider infinitesimal conformal deformations, as a linearization of Definition \ref{def:confequi}.
\begin{definition}
	Suppose $f:V \to \field{R}^n$ is a  realization of a triangulated surface. An infinitesimal deformation $\dot{f}:V \to \field{R}^n$ is \textbf{conformal} if it preserves the length cross ratios. Equivalently, $\dot{f}$ is conformal if there exists $u:V \to \field{R}$ satisfying
	\[
	\langle \dot{f}_j - \dot{f}_i, f_j - f_i \rangle = \frac{u_j + u_i}{2} |f_j - f_i|^2.
	\] 
	We call $u$ the \textbf{scale factor} of $\dot{f}$. In particular, $\dot{f}$ is isometric if $u \equiv 0$.
\end{definition}


As a remark, the scale factor $u$ is a good \emph{intrinsic} parameter to describe conformal deformations for triangulated spheres in $\field{R}^3$.

\begin{proof}[Proof of Proposition \ref{prop:givenu}]
	It is known that as a result of Definition \ref{def:confequi} the space of conformal equivalence class is of dimension $|E|-|V|$ intrinsically, i.e. independent of the immersion \cite{Bobenko2010} \cite[Lemma 18]{Lam2016}. It implies that for a triangulated sphere in $\field{R}^3$, the space of infinitesimal conformal deformations $C$, including Euclidean motions, is of dimension at least $3|V|- (|E|-|V|)=|V|+6$. The inequality is strict if the constraints become linearly dependent and this property depends on the immersion. In the following, we argue that it is indeed an equality if we assume $f$ is infinitesimally rigid.
	
	The map from the space of infinitesimal conformal deformations in $\mathbb{R}^3$ to the space of scale factors $u$ is a linear map $T:C \to \field{R}^{|V|}$.
	
	If the surface is infinitesimally rigid, then $\Ker(T)$ consists of Euclidean motions only and is of dimension $6$. It implies $\mbox{rank} \,T = |V|$ and $\dim C = |V| + 6$. In particular, $T$ is surjective. Thus, given any $u: V \to \mathbb{R}$, there exists an infinitesimal conformal deformation in $\mathbb{R}^3$ with scale factor $u$ and it is unique up to Euclidean motions.  
\end{proof}

\begin{proposition}[Gluck \cite{Gluck1975}] \label{prop:gluck}
	Almost all simply connected closed triangulated surfaces in $\mathbb{R}^3$ are infinitesimally rigid.
\end{proposition}

\begin{corollary}  \label{prop:intrinsic}
	For almost all simply connected closed surfaces in $\mathbb{R}^3$, there exists an infinitesimal conformal deformation unique up to Euclidean motions for any function $u:V \to \field{R}$ as the scale factor.
\end{corollary}
\begin{proof}
	It follows from Proposition \ref{prop:givenu} and \ref{prop:gluck}. 
\end{proof} 

As a remark, all triangulated surfaces except tetrahedra admit non-trivial infinitesimal conformal deformations regardless of their infinitesimal rigidity and genus. This can be observed simply by counting: A closed triangulated surface in $\field{R}^3$ of genus $g$ with $V$ vertices has $3V$ degrees of freedom. Notice that the length cross-ratios satisfy $\Pi_j \mbox{lcr}_{ij} =1$ for every vertex $i$. In order to preserve the conformal structure $\mbox{lcr}:E \to \field{R}$ infinitesimally, there are $E-V = 2V-6+6g$ linear constraints. Hence generally the space of infinitesimal conformal deformations in space is at least $3V-(E-V) = V+6-6g$. If the surface is a tetrahedra, then this number is equal to $10$ which coincides with the dimension of the space of M\"{o}bius transformations.

\section{Infinitesimal isometric deformations of $\field{S}^n $ in $\field{R}^{n+1}$}\label{sec:isoconfomal}

We show that every infinitesimal conformal deformation corresponds to an infinitesimal isometric deformation via stereographic projection. Therefore one can apply techniques from the theory of infinitesimal rigidity to that of infinitesimal conformal deformations, such as \emph{rigidity matrices}.

For an inscribed triangulated surface, i.e. whose vertices lie on the unit sphere $S^n$, we establish a correspondence between its infinitesimal \emph{conformal} deformations tangent to the sphere and infinitesimal \emph{isometric} deformations in $\field{R}^{n+1}$. 
\begin{proposition}\label{prop:projection}
	Given an inscribed triangulated surface $f:V \to S^n \subset \field{R}^{n+1}$, its space of infinitesimal isometric deformations in $\field{R}^{n+1}$ is isomorphic to the space of infinitesimal conformal deformations tangent to the sphere $S^n$.
\end{proposition}
\begin{proof}
	Suppose $v:V \to \field{R}^{n+1}$ is an infinitesimal isometric deformation of $f$. Then its projection  $v^T$ to the tangent space of the sphere is
	\[
	v^{T}:= v - \langle v,f \rangle f.
	\]
	Since $v$ is isometric, i.e. $\langle v_j - v_i, f_j - f_i \rangle=0$, we have
	\begin{align*}
		&\langle v^T_j - v^T_i, f_j - f_i \rangle \\ =& -\langle \langle v_j,f_j \rangle f_j - \langle v_i,f_i \rangle f_i, f_j - f_i \rangle \\
		=&-\frac{1}{2}\langle (\langle v_j,f_j \rangle + \langle v_i,f_i \rangle) (f_j -  f_i) + (\langle v_j,f_j \rangle - \langle v_i,f_i \rangle) (f_j +  f_i), f_j - f_i \rangle \\
		=& -\frac{1}{2}\langle (\langle v_j,f_j \rangle + \langle v_i,f_i \rangle) |f_j -  f_i|^2.
	\end{align*}
	Hence $v^T$ is an infinitesimal conformal deformation with scale factor $-\langle v,f \rangle$. 
	
	We are going to show that such a projection from infinitesimal isometric deformations to infinitesimal conformal deformations is bijective. Assume $v^T\equiv 0$. Then, we have
	\[
	v_i = a_i f_i
	\] 
	for some $a:V \to \field{R}$. Because $v$ is an infinitesimal isometric deformation, we have
	\[
	a_i = \langle v_i, f_i \rangle = -\langle v_j, f_j \rangle = -a_j \quad \forall e_{ij} \in E.
	\]
	Consider the three vertices of any triangle, such condition is satisfied if and only if $a\equiv 0$.
	Hence, $v\equiv 0$ and the projection is injective.
	
	On the other hand, suppose $w$ is an infinitesimal conformal deformation tangent to $S^n$ with scale factor $u$. We define an infinitesimal deformation by
	\[
	v := w - u f.  
	\]
	Then,
	\begin{align*}
		&\langle v_j-v_i ,f_j-f_i \rangle \\
		=& \langle w_j -w_i, f_j -f_i  \rangle - \langle u_j f_j -u_i f_i, f_j -f_i  \rangle \\
		=& \frac{u_i + u_j}{2} |f_j-f_i|^2 - \frac{1}{2}\langle (u_j -u_i) (f_j+f_i) + (u_j+u_i) (f_j-f_i), f_j -f_i  \rangle \\
		=& 0
	\end{align*}
	which implies $v$ is an infinitesimal isometric deformation and $v^{T} = w$. Hence the projection is bijective. 
\end{proof}

The following is the infinitesimal version of the M\"{o}bius invariance of conformal equivalence.
\begin{lemma} \label{lem:mobinf}
	Let $\dot{f}$ be an infinitesimal conformal deformation of $f:V \to \field{R}^n$. Then for every M\"{o}bius transformations $\Phi$, the infinitesimal deformation $d\Phi(\dot{f})$ of $\Phi \circ f$ is conformal.
\end{lemma}
\begin{proof}
	Since M\"{o}bius transformations are generated by Euclidean transformations and inversions, it suffices to consider the inversion under the unit sphere, $\Phi (f) = -f/|f|^2$. Suppose $\dot{f}$ is an infinitesimal conformal deformation of $f$ with scale factor $u:V \to \field{R}$. We have
	\[
	d\Phi(\dot{f})= - \frac{\dot{f}}{|f|^2} + \frac{2 \langle \dot{f}, f \rangle}{|f|^4} f.
	\]
	By direct computation, we get
	\begin{align*}
		& \langle d\Phi(\dot{f}_j)- d\Phi(\dot{f}_i), \Phi(f_j) - \Phi(f_i) \rangle \\
		=& (u_i - \frac{2 \langle \dot{f}_i, f_i \rangle}{|f_i|^2} + u_j - \frac{2 \langle \dot{f}_j, f_j \rangle}{|f_j|^2}) |\Phi(f_j) - \Phi(f_i)|^2 
	\end{align*}
	which implies $d\Phi(\dot{f})$ is an infinitesimal conformal deformation of $\Phi \circ f$. 
\end{proof}

\begin{proof}[Proof of Proposition \ref{thm:isoconf}]
	With the fact that the stereographic projection is a M\"{o}bius transformation, Proposition \ref{thm:isoconf} follows from Lemma \ref{prop:projection} and Lemma \ref{lem:mobinf}. 
\end{proof}

\section{Infinitesimal conformal deformations in $\field{R}^3$} \label{sec:5}

We represent infinitesimal conformal deformations in terms of scale factor $u:V\to \mathbb{R}$ describing the change in edge lengths and vector field $Z:F \to \mathbb{R}^3$ that determines the rotation of faces. The scale factor $u$ alone induces a change in the Gaussian curvature \cite{Glickenstein2011}. To realize such a change in $\mathbb{R}^3$, e.g. from a flat plane to a cone, one has to rotate the faces in a certain way such that the surface is not teared apart. This section is to establish the relation between scale factor $u$ and rotation vector field $Z$ as in the following theorem, which will be a cornerstone for Theorem \ref{thm:givenrho}. 

\begin{theorem} \label{thm:uZ} 
	Let $f:V \to \field{R}^3$ be a  triangulated surface with face normal $N:F \to \mathbb{S}^2$ and $\dot{f}:V \to \field{R}^3$ be an infinitesimal conformal deformation with scale factor $u:V \to \field{R}$. Then there exists a unique vector field defined on faces $Z: F \to \field{R}^3$  satisfying
	\begin{align} \label{eq:infinlambda}
		\begin{aligned}
			\dot{f}_j - \dot{f}_i &= \frac{u_i + u_j}{2} (f_j - f_i) + (f_j - f_i) \times (Z_{ijk} + \frac{\cot \angle jki}{2} (u_j - u_i) N_{ijk})  \\
			&= \frac{u_i + u_j}{2} (f_j - f_i) + (f_j - f_i) \times (Z_{ilj} + \frac{\cot \angle ilj}{2} (u_i - u_j) N_{ilj})
		\end{aligned}
	\end{align}
	where $\{ijk\},\{ilj\}$ are the left and the right faces of the oriented edge $e_{ij}$. In particular, eq. \eqref{eq:infinlambda} implies that the functions $u, Z$ satisfy
	\begin{align}
		(f_j-f_i) \times  ((Z_{ijk} - Z_{ilj}) + (u_j-u_i) (\frac{\cot \angle jki}{2}N_{ijk}+\frac{\cot \angle {ilj}}{2}N_{ilj})) &= 0 \label{eq:imDD}
	\end{align}
	and the change in dihedral angles $\dot{\alpha}$ is given by
	\begin{align} \label{eq:changedi}
		\langle f_j -f_i,Z_{ijk} - Z_{ilj}\rangle  &= \dot{\alpha}_{ij}|f_j-f_i|.
	\end{align}
	Conversely, if the triangulated surface is simply connected and functions $u : V \to \field{R}$, $Z:F \to \field{R}^3$ satisfy \eqref{eq:imDD}, then there exists an infinitesimal conformal deformation $\dot{f}$ satisfying \eqref{eq:infinlambda} unique up to translations.
\end{theorem}
The rest of this section is to prove the theorem above.

Generally, every infinitesimal deformation $\dot{f}$ can be written as
\[
\dot{f}_j - \dot{f}_i = \sigma_{ij} (f_j-f_i)  + W_{ij} \times (f_j-f_i) \ 
\]
where $\sigma_{ij}=\sigma_{ji} \in \mathbb{R}$ and $W_{ij} = W_{ji} \in \mathbb{R}^3$. Note that $W$ is unique up an additive multiple of $f_j - f_i$. This additive constant will be normalized within each of its neighboring faces separately. The difference of the additive constants from neighboring faces will yield the change in the dihedral angle between the face normals. 

Suppose $\dot{f}$ is conformal with scale factor $u$, we have
\[
\sigma_{ij} = \langle \dot{f}_j - \dot{f}_i, f_j - f_i \rangle  / |f_j - f_i|^2 = \frac{u_i + u_j}{2}.
\]

It remains to investigate the dependence of $W$ on $u$. We focus on a triangle $\{ijk\}$ and its edges first. We denote $N_{ijk}$ the face normal. As mentioned, $W_{ij}$ is unique up an additive multiple of $f_j - f_i$. We will normalize $W_{ij}$ with respect to $\{ijk\}$ and write the normalization as $W_{ij,k}$. For the moment we decompose $W_{ij}$ into two components
\[
W_{ij} = \omega_{ij,k} N_{ijk} + Y_{ij,k} 
\]
for some $\omega_{ij,k} \in \mathbb{R}$ and  $Y_{ij,k}  \perp N_{ijk}$. We define
\[
Y_{ijk}:= N_{ijk} \times \dot{N}_{ijk}
\]
where $\dot{N}$ is the change in the face normal. Since $\langle \dot{N}, N \rangle = 0$, we have
\[
\dot{N}_{ijk} = Y_{ijk} \times N_{ijk}.
\]
Note that $Y_{ijk}$ and $Y_{ij,k}$ are perpendicular to $N_{ijk}$ and
\[
\langle Y_{ijk} \times N_{ijk}, f_j-f_i \rangle = \langle \dot{N}_{ijk}, f_j - f_i \rangle = - \langle N_{ijk}, \dot{f}_j - \dot{f}_i \rangle = \langle Y_{ij,k} \times N_{ijk} , f_j - f_i \rangle
\]
which implies $Y_{ijk}-Y_{ij,k}$ is a multiple of $f_j-f_i$. Similarly, we can deduce
$Y_{ijk}-Y_{jk,i}$ is a multiple of $f_k-f_j$ and $Y_{ijk}-Y_{ki,j}$ is a multiple of $f_i-f_k$. We then define
\[
W_{ij,k}:= \omega_{ij,k} N_{ijk} + Y_{ijk} 
\]
to be a normalization of $W_{ij}$ with respect to $\{ijk\}$ since $W_{ij,k} - W_{ij}$ differ by a multiple of $f_j - f_i$. Similarly we have $W_{jk,i}$ and $W_{ki,j}$. The geometric meaning of $Y_{ijk}$ is clear. It describes the rotation of the face normal. The next step is to derive a relation between $\omega$ and scale factor $u$.

The closeness condition
\begin{align*}
	0 =& (\dot{f}_j - \dot{f}_i ) + (\dot{f}_k - \dot{f}_j )+(\dot{f}_i - \dot{f}_k )
\end{align*}
implies
\begin{align}\label{eq:con3final} \begin{aligned}
		0=& \quad \sigma_{ij} (f_j-f_i) + \sigma_{jk} (f_k -f_j) + \sigma_{ki} (f_i - f_k) + \omega_{ij,k} N_{ijk} \times  (f_j -f_i)\\ &+ \omega_{jk,i} N_{ijk}  \times (f_k-f_j) + \omega_{ki,j} N_{ijk} \times(f_i -f_k).
	\end{aligned}
\end{align}
Note that $f_j-f_i \in \mbox{span}\{N_{ijk}\times (f_k-f_j),N_{ijk}\times (f_i-f_k)\}$. In fact
\begin{align*}
	f_j-f_i &= \cot(\angle ijk) N_{ijk} \times (f_i-f_k)-\cot(\angle kij) N_{ijk} \times (f_k-f_j).
\end{align*}
Substituting it into \eqref{eq:con3final} implies
\begin{align*}
	0 =& \sum \big(\omega_{ij,k}+(\sigma_{jk} - \sigma_{ki}) \cot \angle jki \big) N_{ijk} \times (f_j-f_i).
\end{align*}
Since $N_{ijk}\times (f_j-f_i)$,$N_{ijk}\times  (f_k-f_j)$ and $N_{ijk}\times  (f_i-f_k)$ span an affine plane and 
\[
N_{ijk}\times (f_j-f_i)+N_{ijk}\times(f_k-f_j)+N_{ijk}\times (f_i-f_k)=0,
\]
there exists a unique number $\omega_{ijk}$ such that
\begin{align*}
	\omega_{ijk} =& \omega_{ij,k}+ (\sigma_{jk} - \sigma_{ki}) \cot \angle jki  \\ =& \omega_{jk,i}+(\sigma_{ki} - \sigma_{ij})  \cot \angle kij \\ =&\omega_{ki,j}+(\sigma_{ij} - \sigma_{jk}) \cot \angle ijk.
\end{align*}
Because $\sigma_{ij}=\frac{u_i+u_j}{2}$ we have
\begin{align*}
	\omega_{ij,k}&=\omega_{ijk} -\frac{\cot \angle jki}{2} (u_j - u_i).
\end{align*}
Thus,
\[
W_{ij,k} = (\omega_{ijk}-\frac{\cot \angle jki}{2} (u_j - u_i)) N_{ijk} + Y_{ijk}= -Z_{ijk}  -\frac{\cot \angle jki}{2} (u_j - u_i) N_{ijk}.
\]
Here $Z_{ijk}:=-(\omega_{ijk} N_{ijk} +Y_{ijk})$ is associated to the face and describes the its infinitesimal rotation. 

Similarly consider the other neighboring face $\{ilj\}$ that contains $\{ij\}$, the normalization of $W_{ji}=W_{ij}$ is
\[
W_{ji,l} = -Z_{ijk}  -\frac{\cot \angle jki}{2} (u_j - u_i) N_{ijk}.
\]
\begin{lemma}
	The dihedral angle $\alpha$ between the normal is defined by
	\[
	\sin \alpha_{ij} = \langle N_{ijk} \times N_{ilj}, \frac{f_j-f_i}{|f_j-f_i|} \rangle.
	\]
	Its change under an infinitesimal conformal deformation satisfies
	\[
	\dot{\alpha}_{ij} = \langle Z_{ijk}-Z_{ilj} , \frac{f_j-f_i}{|f_j-f_i|} \rangle = \langle W_{ji,l} - W_{ij,k} , \frac{f_j-f_i}{|f_j-f_i|} \rangle
	\]
	where $\{ijk\},\{ilj\}$ are the left and the right face of the oriented edge from $i$ to $j$. Since $W_{ji,l} - W_{ij,k}$ differ only by a multiple of $f_j-f_i$, we indeed have
	\[
	W_{ji,l} - W_{ij,k} = \dot{\alpha}_{ij} \frac{f_j-f_i}{|f_j-f_i|}
	\]
\end{lemma}
\begin{proof}
	Differentiating $\sin \alpha$ yields
	\begin{align*}
		\dot{\alpha}_{ij} \cos \alpha_{ij} =& \langle \dot{N}_{ijk} \times N_{ilj}+N_{ijk} \times \dot{N}_{ilj}, \frac{f_j-f_i}{|f_j-f_i|} \rangle \\
		=& \langle \big(N_{ijk} \times Z_{ijk} \big) \times N_{ilj}+N_{ijk} \times \big(N_{ilj} \times Z_{ilj}\big), \frac{f_j-f_i}{|f_j-f_i|} \rangle \\
		=& \cos \alpha_{ij} \langle Z_{ijk} - Z_{ilj} , \frac{f_j-f_i}{|f_j-f_i|} \rangle.
	\end{align*}
	Hence
	\[
	\dot{\alpha}_{ij} = \langle Z_{ijk} - Z_{ilj} , \frac{f_j-f_i}{|f_j-f_i|} \rangle.
	\] 
\end{proof}
The above lemma implies the function $Z:F\to\mathbb{R}^3$ satisfy a compatibility condition that
\begin{equation} \label{eq:vectorexact}
	Z_{ijk}-Z_{ilj}=-(u_j-u_i)(\frac{\cot \angle jki}{2}N_{ijk}+\frac{\cot \angle {ilj}}{2}N_{ilj})+\dot{\alpha}_{ij} \frac{f_j-f_i}{|f_j-f_i|} .
\end{equation}
We decompose this equation into two components. One is parallel to $f_j-f_i$
\begin{align*}
	\langle f_j-f_i,Z_{ijk}-Z_{ilj}\rangle  = \dot{\alpha}_{ij}|f_j-f_i|
\end{align*}
and the other is perpendicular to $f_j-f_i$ 
\begin{align*}
	(f_j-f_i) \times  ((Z_{ijk} - Z_{ilj}) + (u_j-u_i) (\frac{\cot \angle jki}{2}N_{ijk}+\frac{\cot \angle {ilj}}{2}N_{ilj})) = 0
\end{align*}
as stated in Theorem \ref{thm:uZ}.

Conversely, given $Z: F \to \field{R}^3$ satisfying  Equation \eqref{eq:vectorexact} for some $u:V \to \field{R}$, one could immediately check that the 1-form
\begin{align*}
	\eta(e_{ij}) &= \frac{u_i + u_j}{2} (f_j-f_i) + (f_j-f_i) \times (Z_{ijk} + \frac{\cot \angle jki}{2} \, (u_j - u_i) N_{ijk}) \\
	&= \frac{u_i + u_j}{2} (f_j-f_i) + (f_j-f_i) \times (Z_{ilj} + \frac{\cot \angle ilj}{2} \, (u_i - u_j) N_{ilj})
\end{align*}
is well defined and is closed. If the triangulated surface is simply connected, then there exists an infinitesimal conformal deformation $\dot{f}:V \to \field{R}^3$ unique up to translations such that
\[
\dot{f}_j - \dot{f}_i = \eta(e_{ij}).
\]

\section{Angular velocity equation}\label{sec:6}

The derivation in the previous section yields an interesting equation. 

\begin{proposition}[Angular velocity equation]\label{thm:ave}
	Under an infinitesimal conformal deformation with scale factor $u:V \to \field{R}$, the change in dihedral angles $\dot{\alpha}$ satisfy for every interior vertex $i$
	\[
	\sum_j \dot{\alpha}_{ij} \frac{f_j-f_i}{|f_j-f_i|} = \sum_j (u_j-u_i)(\frac{\cot \angle jki}{2}N_{ijk}+\frac{\cot \angle {ilj}}{2}N_{ilj}) = \sum_{jk} \dot{\beta}_{ijk}  N_{ijk}
	\]
	where $\{jk\}$ is an edge of a face $\{ijk\}$ and $\dot{\beta}_{ijk}$ denotes the change in $\angle ijk$.
\end{proposition}

\begin{proof}
	It follows from Equation \eqref{eq:vectorexact} that for every interior vertex $i$
	\[
	\sum_j \big( \dot{\alpha}_{ij} \frac{f_j-f_i}{|f_j-f_i|} -(u_j-u_i)(\frac{\cot \angle jki}{2}N_{ijk}+\frac{\cot \angle {ilj}}{2}N_{ilj}) \big) = \sum_j (Z_{ijk}-Z_{ilj})= 0
	\]
	Furthermore
	\[
	\cos \beta_{kij} = \frac{\langle f_j-f_i, f_k-f_i\rangle}{|f_j-f_i| |f_k-f_i|}. 
	\]
	Differentiating both sides yields
	\begin{align*}
		-\dot{\beta}_{kij} \sin \beta_{kij} 
		=& -\sin \beta_{kij} ((u_j-u_i)\frac{\cot \angle jki}{2} + (u_k-u_i)\frac{\cot \angle ijk}{2} )
	\end{align*}
	and thus
	\[
	\dot{\beta}_{kij} = (u_j-u_i)\frac{\cot \angle jki}{2} + (u_k-u_i)\frac{\cot \angle ijk}{2}.
	\] 
\end{proof}

For every infinitesimal isometric deformation, we have $u \equiv 0$ and hence obtain the standard identity \cite[Lemma 28.2]{Pak2010}
\[
\sum_j \dot{\alpha}_{ij} \frac{f_j-f_i}{|f_j-f_i|}=0   \quad \forall i \in V_{int}
\]
which is related to infinitesimal deformations of spherical polygons with fixed edge lengths \cite{Schlenker2007}.

\begin{corollary}[\cite{Springborn2008,Lam2016a}]\label{cor:planarconf}
	Suppose an infinitesimal conformal deformation of an immersed triangulated surface in the plane has scale factor $u: V \to \field{R}$. Then the function $u$ is a discrete harmonic function with respect to the cotangent Laplacian, i.e. for every interior vertex $i$
	\[
	\sum_j (\cot \angle jki+ \cot \angle {ilj})(u_j-u_i) =0
	\]
	and the change in dihedral angle $\dot{\alpha}$ satisfies
	\[
	\sum_j \dot{\alpha}_{ij} \frac{f_j-f_i}{|f_j-f_i|} = 0.
	\]
\end{corollary}
\begin{proof}
	It follows immediately from Proposition \ref{thm:ave} by considering the component in plane and the normal component separately.
\end{proof}

\section{Mean curvature half density}\label{sec:7}

We are interested in parameterizing the space of infinitesimal conformal deformations in terms of the extrinsic geometry. Under an infinitesimal conformal deformation, we consider the \emph{change in mean curvature half-density} $\rho:V_{int} \to  \field{R}$
\[
\rho_i := \frac{1}{2} \sum_j \dot{\alpha}_{ij}|f_j - f_i|.
\]
This formula is closely related to the mean curvature arisen from Steiner's tube formula:
\begin{equation} \label{eq:tube}
	i \in V \mapsto \frac{1}{2} \sum_j \alpha_{ij} |f_j - f_i|
\end{equation}
Notice that $\rho$ is not exactly the derivative of equation \eqref{eq:tube} unless the deformation is isometric. In fact, equation \eqref{eq:tube} is a discrete analogue of $ H |df|^2$ while $\rho$ is analogous to the change in mean curvature half density $(\frac{d}{dt} \,H|df| )  |df|$. Here one might heuristically think of $\alpha$ as $H|df|$. In the smooth theory, mean curvature half density is used to parametrize conformal deformations of surfaces in space \cite{Richter1997,Crane2011}. 

As a corollary of Theorem \ref{thm:uZ}, we obtain the famous Schl\"{a}fli formula. It is a necessary condition for $\rho$ and later it turns out to be sufficient in non-degenerate cases. 

\begin{corollary}{(Scalar Schl\"{a}fli  Formula)}\label{prop:schafli}
	For every infinitesimal conformal deformation of a closed triangulated surface $f:V \to \field{R}^3$, we have 
	\[
	\sum_i \rho_i = \sum_{ij} \dot{\alpha}_{ij} |f_j-f_i| = 0
	\]
	where $\dot{\alpha}$ is the change in the dihedral angles.
\end{corollary}
\begin{proof}
	Given an infinitesimal conformal deformation, Theorem \ref{thm:uZ} shows that there exists $(u,Z) \in \field{R}^v \times \field{R}^{3F}$ such that
	\[
	\rho_i  = \frac{1}{2}\sum_{j} \dot{\alpha}_{ij} |f_j-f_i| = \frac{1}{2}\sum_{j} \langle f_j-f_i,Z_{ijk}-Z_{jki}\rangle. 
	\]
	Hence we have
	\[
	\sum_{i} \rho_i = \sum_{ij \in E} \dot{\alpha}_{ij} |(f_j-f_i)| =  \sum_{ijk} \langle f_j-f_i+f_k-f_j+f_i-f_k,Z_{ijk}\rangle =0.
	\] 
\end{proof}

\begin{remark}
	Scalar Schl\"{a}fli  Formula holds for general infinitesimal deformations of closed triangulated surfaces \cite{Pak2010}. 
\end{remark}

\section{Discrete Dirac operator}\label{sec:8}

A heuristic counting of the degrees of freedom indicate that the change in mean curvature half-density is a good candidate to parametrize infinitesimal conformal deformations extrinsically. Notice that:
\begin{enumerate}
	\item The space of infinitesimal conformal deformations of a triangulated sphere is at least $3|V| - (|E|-|V|) = |V|+6$. 
	\item Similarity transformations are trivial conformal deformations that preserve dihedral angles. The space of similarity transformations has dimension $7$.
	\item The space of functions $\rho: V \to \mathbb{R}$ with $\sum_i \rho_i =0$ is of dimension $|V|-1= (|V|+6) - 7$
\end{enumerate}
Hence it is reasonable to ask if, given a function $\rho: V \to \mathbb{R}$ with $\sum_i \rho_i =0$, does there exist an infinitesimal conformal deformation unique up to a similarity transformation with $\rho$ as the change in mean curvature half-density? The answer is positive as stated in Theorem \ref{thm:givenrho}. We follow the proof in the smooth theory \cite{Richter1997} by developing a discrete Dirac operator $\D$ for closed triangulated surfaces.

\begin{definition} \label{def:dirac}
	Given a realization of a triangulated surface $f:V \to \mathbb{R}^3$, we associate each interior edge the two dimensional subspace in $\mathbb{R}^3$ that is perpendicular to $f_j-f_i$. It forms a normal bundle $\mathcal{N}_f$ over the edges. A \emph{section} of $\mathcal{N}_f$ is a function $W:E \to \mathcal{N}_f$ such that $W_{ij}=W_{ji}$ is perpendicular to $f_j-f_i$. The space of sections are denoted by $\Gamma(\mathcal{N}_f)$, which is of dimension $2|E|$ as a vector space. We write $f_{ijk}$ to represent the circumcenter of the face $\{ijk\}$ and have
	\[
	f_{ijk}-f_{ilj}=(\frac{\cot \angle jki}{2}N_{ijk}+\frac{\cot \angle ilj}{2}N_{ilj}) \times (f_j-f_i)
	\]
	We define the \textbf{discrete Dirac operator} 
	\begin{align*}
		\D:\field{R}^{V} \times \field{R}^{3F} &\to \field{R}^{V_{int}} \times \Gamma(\mathcal{N}_f) \\
		(u,Z) &\mapsto (\rho,U)		
	\end{align*}
	where
	\begin{align*}
		\rho_i &= \frac{1}{2}\sum_j \langle f_j-f_i,Z_{ijk}-Z_{ilj}\rangle,  \\
		U_{ij}&=-(f_j-f_i)\times (Z_{ijk}-Z_{ilj})+(f_{ijk}-f_{ilj})(u_j-u_i) = U_{ji}.
	\end{align*}
	and $\{ijk\}$,$\{jil\}$ are the left and the right triangles of the oriented edge from $i$ to $j$.
\end{definition}

This definition follows from Theorem \ref{thm:uZ}. Its name is motivated from the smooth theory \cite{Crane2011}. In the following we derive the adjoint of the discrete Dirac operator on closed triangulated surfaces.

\begin{definition} We denote $\langle \cdot,\cdot \rangle$ the Euclidean product in $\mathbb{R}^3$. On a closed triangulated surface, we further define inner products on $\mathbb{R}^{V} \times \Gamma(\mathcal{N}_f)$ and $\mathbb{R}^{V} \times \mathbb{R}^{3F}$ respectively. For $(\alpha,W),(\tilde{\alpha},\tilde{W}) \in \mathbb{R}^{V} \times \Gamma(\mathcal{N}_f)$ where $W_{ij}=W_{ji} $ is a vector perpendicular to $f_j-f_i$, then
	\[
	\big( (\alpha,W), (\tilde{\alpha}, \tilde{W}) \big) := \sum_{i \in V} \alpha_i \tilde{\alpha}_i + \sum_{ij \in E} \langle W_{ij}, \tilde{W}_{ij} \rangle.
	\]
	The sums are respectively over all vertices and over all unoriented edges. For $(\beta,Y),(\tilde{\beta},\tilde{Y}) \in \mathbb{R}^{V} \times \mathbb{R}^{3F}$ where $Y,\tilde{Y}:F \to \mathbb{R}^3$,
	\[
	\big( (\beta,Y), (\tilde{\beta}, \tilde{Y}) \big) := \sum_{i \in V} \beta_i \tilde{\beta}_i + \sum_{ijk \in F} \langle Y_{ijk}, \tilde{Y}_{ijk} \rangle.
	\]
\end{definition}

With the inner products, the adjoint discrete Dirac operator $\D^*:\mathbb{R}^{V} \times \Gamma(\mathcal{N}_f) \to \field{R}^{V} \times \mathbb{R}^{3F}$ is defined as the map satisfying 
\[
\big((u,Z),\D^*(\alpha,W)\big) = \big(\D(u,Z),(\alpha,W)\big)
\]
for $(u,Z) \in \field{R}^{V} \times \mathbb{R}^{3F}$ and $(\alpha,W) \in \mathbb{R}^{V} \times \Gamma(\mathcal{N}_f)$.

\begin{proposition}\label{prop:adD}
	The adjoint discrete Dirac operator has an explicit form: 
	\begin{align*}
		\D^*: \mathbb{R}^{V} \times \Gamma(\mathcal{N}_f) &\to \field{R}^{V} \times \field{R}^{3F} \\
		(\alpha, W) &\mapsto (\tilde{\rho},Y)
	\end{align*}
	where $\tilde{\rho}:V \to \mathbb{R}$ and $Y:F\to \mathbb{R}^3$ are defined by
	\begin{align*}
		\tilde{\rho}_i =& -\sum_{j } \langle f_{ijk} - f_{ilj},W_{ij}\rangle \\
		Y_{ijk} =& (f_j-f_i) \times W_{ij}  + (f_k-f_j) \times W_{jk} + (f_i-f_k) \times W_{ki} \\ & + \frac{\alpha_i + \alpha_j}{2} (f_j-f_i) + \frac{\alpha_j + \alpha_k}{2} (f_k-f_j) + \frac{\alpha_k + \alpha_i}{2} (f_k-f_j)
	\end{align*}
	and $f_{ijk}$, $f_{ilj}$ are the circumcenters of the left and the right faces of the oriented edge from $i$ to $j$.
\end{proposition}
\begin{proof}
	Direct computations yield
	\begin{align*}
		((u,Z),\D^*(\alpha,0)) =& \sum_{i \in V} \frac{\alpha_i}{2} \sum_j \langle f_j-f_i,Z_{ijk}-Z_{ilj}\rangle  \\
		=& \sum_{ijk} \langle Z_{ijk}, \frac{\alpha_i \!  + \! \alpha_j}{2} (f_j-f_i) \! +\!   \frac{\alpha_j \! + \!  \alpha_k}{2} (f_k-f_j) \! + \!  \frac{\alpha_k \! +  \alpha_i}{2} (f_i -f_k) \\
		((u,0),\D^*(0,W)) =& \sum_{ij \in E} \langle (f_{ijk}-f_{ilj}) (u_j-u_i), W_{ij}\rangle  \\
		=& -\sum_{i \in V}  u_i \sum_{j} \langle f_{ijk}-f_{ilj},W_{ij}\rangle , \\                                    
		((0,Z),\D^*(0,W)) =& \sum_{ij \in E} \langle -(f_j-f_i) \times Z_{ijk}-Z_{ilj}, W_{ij}\rangle  \\
		=& \sum_{ijk} -\langle Z_{ijk}, W_{ij} \times (f_j-f_i) + W_{jk} \times (f_k-f_j)+ W_{ki} \times (f_i -f_k)\rangle
	\end{align*}
	By linearity, we obtain the formula as stated. 
\end{proof}

\begin{lemma}
	On a closed triangulated surface, we have $3|E| = 2 |F|$ and
	\begin{align*}
		\dim(\Im \D) &= \dim(\Im\D^*), \\
		\dim(\Ker \D) &= \dim(\Ker \D^*), \\
		(\Ker D^*)^{\bot} &= \Im D.
	\end{align*}
\end{lemma}

\begin{lemma} The kernel of $D^*$ contains constant functions
	\begin{equation*}
		\{ (\alpha, W^{\dagger}) \in \field{R}^{V} \times \Gamma(\mathcal{N}_f) | \alpha \in \field{R}, W \in \field{R}^3\} \subset \Ker \D^*
	\end{equation*}
	where $W_{ij}^{\dagger}$ is the component of $W_{ij}$ orthogonal to $f_j-f_i$. In particular \[ \dim(\Ker \D^*) \geq 4. \]
\end{lemma}
\begin{proof}
	Let $\alpha$ be a constant function on vertices and $W$ be a constant vector in $\mathbb{R}^3$. We write $(\tilde{\rho},Y):= \D^*(\alpha, W^{\dagger})$ and have
	\begin{align*}
		\tilde{\rho}_i =& -\sum_{j} \langle f_{ijk}-f_{ilj},W\rangle =0, \\
		Y_{ijk} =& 
		(f_j-f_i) \times W  + (f_k-f_j) \times W + (f_i-f_k)) \times W \\& + \frac{\alpha + \alpha}{2} (f_j-f_i) + \frac{\alpha+ \alpha}{2} (f_k-f_j) + \frac{\alpha + \alpha}{2} (f_i-f_k))  \\
		=&0.																	
	\end{align*}
	
	If $W^\dagger\equiv 0$ for a constant vector $W$, it implies $W$ is parallel to all edges of $f$. Since $f$ is non-degenerate, we have $W\equiv 0$. Hence, the nullity of $\D^*$ is at least 4. 
\end{proof} 

\begin{lemma} \label{lemma:sumzero}
	If $\dim(\Ker \D)=4$, then for all $(\rho,U) \in \field{R}^V \times \Gamma(\mathcal{N}_f)$,
	\[
	\sum_{i \in V} \rho_i =0 \text{ and } \sum_{ij\in E} U_{ij} =0 \Leftrightarrow (\rho,U) \in \Im\D.
	\]
\end{lemma} 
\begin{proof}
	Suppose there exists $(u,Z) \in \field{R}^{V+3F}$ such that $D(u,Z) = (\rho,U)$. Then  Proposition \ref{prop:schafli} implies $\sum_{i \in V} \rho_i =0$ while $ \sum_{ij \in E} U_{ij} =0$ follows from direct computation.
	
	Conversely, the assumption implies $\dim(\Ker \D^*)=\dim(\Ker \D)=4$ and thus $\Ker \D^*$ contains constant functions only. Therefore, 
	\[
	\sum_{i \in V} \rho =0 \text{ and } \sum_{ij \in E} U_{ij} =0
	\]
	yields $(\rho,U) \in (\Ker\D^*)^{\perp} = \Im\D$. 
\end{proof}

We are now ready to prove Theorem \ref{thm:givenrho}.
\begin{proof}[Proof of Theorem \ref{thm:givenrho}]
	Notice that infinitesimal translations correspond to $u \equiv 0, Z  \equiv 0$ in Theorem \ref{thm:uZ}. Furthermore infinitesimal rotations are given by $u \equiv 0, Z \equiv const.$ while uniform scaling by $u \equiv const., Z \equiv 0$.
	
	Given a function $\rho:V \to \mathbb{R}$ with $\sum_i \rho_i =0$, then $(\rho,0) \in \Im\D$ and hence there exists $(u,Z) \in \mathbb{R}^{V}\times \mathbb{R}^{3F}$ such that $D(u,Z)= (\rho,0)$. The assumption $\dim \Ker D=4$ implies $u,Z$ are unique up to a constant. Thus with Theorem \ref{thm:uZ}, we conclude that there exists an infinitesimal conformal deformation satisfying
	\[
	\rho_i = \frac{1}{2}\sum_j \dot{\alpha}_{ij} |f_j-f_i|.
	\] 
	The deformation is unique up to a similarity transformation. 
\end{proof}

We finish the discussion on the adjoint discrete Dirac operator with three corollaries. The first is a dual of Theorem \ref{thm:uZ} in the sense that infinitesimal conformal deformations can be described in terms of $D^*$ in the same way as that of $D$. This result should be expected in the smooth theory because the Dirac operator is self-adjoint \cite{Kamberov1998}. 

\begin{corollary} \label{cor:dualuw}
	Let $\dot{f}:V \to \field{R}^3$ be an infinitesimal conformal deformation of a  realization $f:V \to \field{R}^3$ with scale factor $u:V \to \field{R}$. Then there exists an element $W \in \Gamma(\mathcal{N}_f)$, i.e. $W_{ij} \in (f_j-f_i)^{\perp} \subset \mathbb{R}^3$ such that
	\begin{align}\label{eq:uwfdot}
		\dot{f}_j - \dot{f}_i &= \frac{u_i + u_j}{2} (f_j-f_i) + (f_j-f_i) \times W_{ij}
	\end{align}
	Furthermore, we have $D^*(u,W)=(\rho,0)\in \mathbb{R}^V \times \mathbb{R}^{3F}$ where 
	\[
	\rho_i = \frac{1}{2}\sum_j \dot{\alpha}_{ij} |f_j-f_i|
	\]
	and $\dot{\alpha}$ is the change in dihedral angles.
	
	Conversely, if the triangulated surface is simply connected and functions $u,W$ satisfy $D^*(u,W)=(\rho,0)$ for some $\rho$, then there exists an infinitesimal conformal deformation $\dot{f}$ unique up to translation given via \eqref{eq:uwfdot}.
\end{corollary}

\begin{proof}
	Given an infinitesimal conformal deformation, the existence and uniqueness of $W$ is immediate. For every face $\{ijk\}$, we have
	\begin{align*}
		0 =& \dot{f}_j -\dot{f}_i + \dot{f}_j -\dot{f}_i + \dot{f}_j -\dot{f}_i 
	\end{align*}
	which implies $D^*(u,W) = (\rho,0)$ for some $\rho: V_{int} \to \field{R}$. On the other hand, as a result of Section \ref{sec:5}, there exists a vector field $Z \to \field{R}^3$ on faces such that
	\[
	W_{ij} = \big(Z_{ijk} + \frac{\cot \angle jki}{2} \, (u_j - u_i) N_{ijk} \big)^{\dagger} =  \big(Z_{ilj} + \frac{\cot \angle ilj}{2} \, (u_i - u_j) N_{ilj}\big)^{\dagger} = W_{ji}
	\]
	where $\dagger$ denotes the component orthogonal to $(f_j-f_i)$. We denote the circumcenter of $\{ijk\}$ as $f_{ijk}$. We have 
	\begin{align*}
		\frac{1}{2}\sum_{j} \dot{\alpha}_{ij} |f_j-f_i| =& \sum_{j} \langle \dot{\alpha}_{ij} \frac{f_j-f_i}{|f_j-f_i|}, f_{ijk} - f_i \rangle \\
		=& \sum_{j} \big( \langle Z_{ijk} + \frac{\cot \angle jki}{2} \, (u_j - u_i) N_{ijk}, f_{ijk} - f_i \rangle \\ &\quad - \langle Z_{ilj} + \frac{\cot \angle ilj}{2} \, (u_i - u_j) N_{ilj}, f_{ijk} - f_i \rangle \big) \\
		=& -\sum_{j} \langle Z_{ilj} + \frac{\cot \angle ilj}{2} \, (u_i - u_j) N_{ilj}, (f_{ijk} - f_i) - (f_{ilj} - f_i) \rangle \\
		=& \rho_i
	\end{align*} 
	as claimed. 
\end{proof}

In contrast to Proposition \ref{prop:givenu}, not every function $u$ can be realized as the scale factor of some infinitesimal conformal deformation in space for infinitesimally flexible surfaces.

\begin{corollary}
	Suppose $f:V \to \mathbb{R}^3$ is a closed surface with a nontrivial infinitesimal isometric deformation $\bar{f}$ having $\bar{\rho}:V \to \mathbb{R}$ the change in mean curvature half density. Then for any infinitesimal conformal deformation in space with scale factor $u$, we have
	\[
	\sum_i \bar{\rho}_i u_i =0.
	\]
\end{corollary}
\begin{proof}
	Using Corollary \ref{cor:dualuw}, there exists $\bar{W} \in \Gamma(\mathcal{N}_f)$ such that
	\[
	\bar{f}_j - \bar{f}_i = (f_j - f_i) \times \bar{W}_{ij} 
	\]
	with $D^*(0,\bar{W}) = (\bar{\rho}, 0)$. For any other infinitesimal conformal deformation in space, we represent it as $(u,Z)$ as in Theorem \ref{thm:uZ} where $u$ is the scale factor and $Z$ is the rotation field on faces. Its change in mean curvature half density $\rho$ satisfies $D(u,Z)= (\rho,0)$. Then
	\begin{align*}
		\sum_i \bar{\rho}_i u_i = \big( D^*(0,W), (u,Z) \big) = \big( (0,W), D(u,Z) \big) = \big( (0,W), (\rho,0) \big) = 0.
	\end{align*} 
\end{proof}

The third corollary reflects that the discrete Dirac operator is the ``square root" of the cotangent Laplacian. Its proof is straightforward by applying Definition \ref{def:dirac} and Proposition \ref{prop:adD}.

\begin{corollary}
	For any real-valued function $\alpha : V \to \field{R}$, we have $\D \frac{1}{A}\D^*(\alpha,0) = ( \rho, U)$ where
	\begin{align*}
		\rho_i &= -\sum_{j}  (\cot \angle jki + \cot \angle ilj) (\alpha_j - \alpha_i), \\
		U_{ij} &= (\alpha_j-\alpha_i)(N_{ijk}-N_{ilj}).
	\end{align*}
	Here $A:F\to \field{R}$ is the signed area of the corresponding triangle under the realization and $\{ijk\},\{ilj\}$ are the left and the right faces of the oriented edge from $i$ to $j$.
\end{corollary}

\section{Examples} \label{sec:9}

We consider triangulated surfaces with vertices on the unit sphere. We investigate their infinitesimal conformal deformations and the change in mean curvature half density. 

\begin{proposition}[\cite{Lam2016}] \label{prop:infiso}
	Every infinitesimal isometric deformation of a triangulated surface with vertices on the sphere satisfies for every interior vertex $i$
	\[
	\sum_j \dot{\alpha}_{ij} |f_j-f_i| = 0. 
	\] 
\end{proposition}

\begin{example}\label{example:jessen}
	
	Jessen's orthogonal icosahedron \cite{Jessen1967} is combinatorially a regular icosahedron with some edges flipped (Figure \ref{fig:Jessen}). Its vertices lie on a sphere and all dihedral angles are either $\pi/2$ or $3\pi/2$. It is a well known example as a non-convex triangulated sphere that is infinitesimally flexible \cite{Lam2016}. Under the non-trivial infinitesimal isometric deformation, the change in the mean curvature half density vanishes by Proposition \ref{prop:infiso}. Hence $\dim(\Ker\D)>4$.
	\begin{figure}[h]
		\centering
		\includegraphics[width=0.3\textwidth]{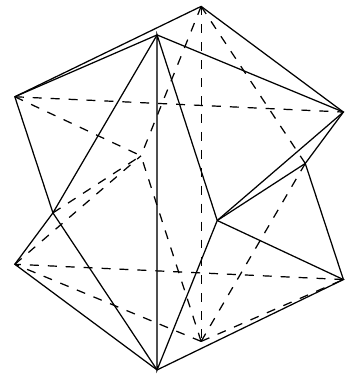}
		\caption{Jessen's orthogonal icosahedron}
		\label{fig:Jessen}
	\end{figure}
\end{example}

\begin{example}\label{example:Bricard}
	Figure \ref{fig:Bricard} shows one of Bricard's octahedra inscribed in a sphere. It is flexible and therefore Proposition \ref{prop:infiso} implies $\dim(\Ker\D)>4$.
	\begin{figure}[h]
		\centering
		\includegraphics[width=0.66\textwidth]{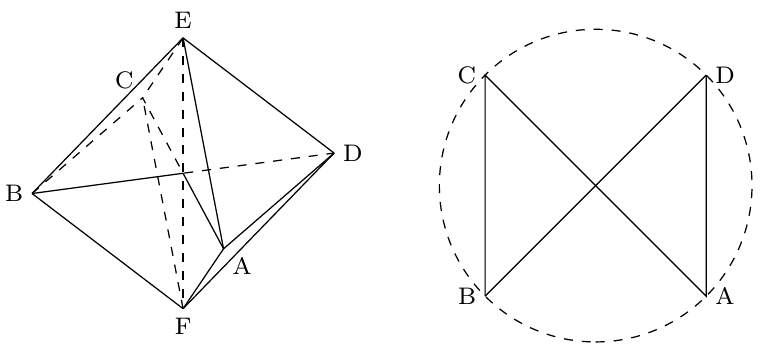}
		\caption{Bricard's octahedron and its top view}
		\label{fig:Bricard}
	\end{figure}
\end{example}

Dehn's theorem states that convex polyhedra are infinitesimally rigid, i.e. there is no infinitesimal isometric deformation other than Euclidean motion. However, one could still consider conformal deformations. For inscribed polyhedra, their infinitesimal conformal deformations are exactly given by normal deformations.

\begin{proposition}
	Suppose $f:V \to S^2$ is an inscribed triangulated surface. For every $u:V \to \field{R}$, the infinitesimal deformation $\dot{f} = u f$ is conformal with scale factor $u$ and the change in mean curvature half-density is
	\[
	\sum_j \dot{\alpha}_{ij} |f_j-f_i| = -\sum_{j} (u_j-u_i) (\frac{\cos d_{ijk} \cot \angle jki}{2} + \frac{\cos d_{ilj}\cot \angle{ilj}}{2}) \quad \forall i \in V_{int}
	\]
	where $d_{ijk}$ is the distance from the origin to the face $\{ijk\}$.
\end{proposition}
\begin{proof}
	Let $u:V \to \field{R}$ and $\dot{f}= u f$. We have
	\begin{align*}
		\langle \dot{f}_j - \dot{f}_i, f_j -f_i \rangle = u_j + u_i - (u_j + u_i) \langle f_j, f_i \rangle = \frac{u_j + u_i}{2} |f_j - f_i|^2.
	\end{align*}
	and hence $\dot{f}$ is an infinitesimal conformal deformation with scalec factor $u$. In terms of Theorem \ref{thm:uZ}, such a deformation is given by $Z:F \to \field{R}^3$
	\[
	Z_{ijk} = -\frac{\cos d_{ijk}}{4 A_{ijk}}  \big(u_i (f_k-f_j) +u_j (f_i-f_k)+u_k (f_j-f_i)\big)
	\]
	where $A_{ijk}$ is the area of the triangle $\{ijk\}$. For every interior vertex $i$, the change in mean curvature half-density is
	\begin{align*}
		\sum_j \dot{\alpha}_{ij} |f_j-f_i| =& \sum_{j} \langle (f_j-f_i), Z_{ijk}- Z_{ilj} \rangle \\
		=&  -\sum_{j} (u_j-u_i) (\frac{\cos d_{ijk} \cot \angle jki}{2} + \frac{\cos d_{ilj}\cot \angle{ilj}}{2})
	\end{align*} 
\end{proof}

\begin{example}
	The regular octahedron has $\dim(\Ker\D)=4$, that means it does not possess non-trivial infinitesimal conformal deformation with vanishing change in mean curvature half-density. On one hand, all its infinitesimal isometric deformations are trivial since it is convex. On the other hand, for each edge $\{ij\}$, the coefficient $(\cos d_{ijk} \cot \angle jki+ \cos d_{ilj}\cot \angle{ilj})/2$ is strictly positive. If there is an infinitesimal conformal deformation such that $\sum_j \dot{\alpha}_{ij} |f_j-f_i| =0 $ for all vertices, then the scale factor $u$ must be constant and the deformation is induced by a similarity transformation since the octahedron is infinitesimally rigid.  
\end{example}

Applying Gluck's argument \cite[Theorem 6.1]{Gluck1975} to the discrete Dirac operator, the set of non-degenerate triangulated spheres in $\mathbb{R}^3$ with $\dim \Ker \D=4$ is the complement of a real algebraic variety in $\mathbb{R}^{3V}$. Hence the set is open and dense as long as the algebraic variety is proper.

\begin{corollary}\label{cor:givenrho}
	For almost all octahedra in $\field{R}^3$, given any $\rho:V \to \field{R}$ with $\sum_i \rho_i=0$, there exists an infinitesimal conformal deformation unique up to a similarity transformation with $\rho$ as the change in mean curvature half density.
\end{corollary}

\section{Triangulated surfaces of high genus} \label{sec:infinhighdisc}

\begin{lemma}\label{lem:hodge}
	For a closed triangulated surface of genus $g$, there exists $2g$ closed dual 1-forms $\omega_1, \omega_2, \dots, \omega_{2g} : \vv{E}^* \to \field{R}$ such that a closed primal 1-form $\eta: \vv{E} \to \field{R}$ is exact if and only if for $k=1,2,\dots, 2g$
	\[
	\sum_{ij \in E} \omega_k(e^*_{ij})\eta(e_{ij}) =0.
	\]
\end{lemma}

The dual 1-forms $\omega_i$ are called \textbf{harmonic 1-forms} and the lemma follows from the Hodge decomposition of discrete differential forms \cite{Desbrun2006a}.

\begin{theorem} \label{thm:highgendis}
	Suppose $f:M \to \field{R}^3$ is a realization of a closed triangulated surface of genus $g$ with $\dim(\Ker \D)=4$. Let $\omega_1,\ldots,\omega_{2g}$ form a basis of harmonic 1-forms and $\mathbf{e}_1 := (1,0,0), \mathbf{e}_2 = (0,1,0)$,$\mathbf{e}_3 = (0,0,1)$ be the standard basis for $\field{R}^3$. Then for every $k=1,2,\ldots,2g$ and $l=1,2,3$, 
	there exists $(u_{kl},Z_{kl}) \in \field{R}^{V} \times \mathbb{R}^{3F}$ unique up to a constant such that $\D (u_{kl},Z_{kl}) = (\alpha_{kl}, W_{kl}) \in \mathbb{R}^{V} \times \Gamma(\mathcal{N}_f)$ where
	\begin{align*}
		\alpha_{kl,i} &= \sum_{j} \omega_k(e^*_{ij}) \langle  \mathbf{e}_l, f_j-f_i \rangle, \\
		W_{kl,ij} &= \omega_k(e^*_{ij}) \mathbf{e}_l \times  (f_j-f_i).
	\end{align*}
	
	Furthermore, given $(u,Z)\in \field{R}^{V+3F}$ with $\D(u,Z)=(\rho,0)$, the $\field{R}^3$-valued primal 1-form
	\begin{align*} 
		\eta(e_{ij}) &:= \frac{u_i + u_j}{2} (f_j-f_i) + (f_j-f_i) \times (Z_{ijk} + \frac{\cot \angle jki}{2} \, (u_j - u_i) N_{ijk}) \\
		&= \frac{u_i + u_j}{2} (f_j-f_i) + (f_j-f_i) \times (Z_{ilj} + \frac{\cot \angle ilj}{2} \, (u_i - u_j) N_{ilj})
	\end{align*} 
	is exact if and only if for every $k=1,2,\ldots,2g$ and $l=1,2,3$,
	\[
	\sum_{i} \rho_i u_{kl,i} =0. 
	\]
\end{theorem}

\begin{proof}
	We first show that $(\alpha_{kl},W_{kl}) \in \Im D$. Since $\omega_k$ is a closed dual 1-form, it satisfies for every vertex $i$
	\[
	\sum_j \omega_k(e^*_{ij}) =0
	\]
	and hence
	\begin{align*}
		\sum_{i \in V} \sum_{j}  \omega_k(e^*_{ij}) \langle \mathbf{e}_l,f_j-f_i\rangle = -2\sum_{i\in V} \big( \langle \mathbf{e}_l,f_i\rangle  \sum_{j} \omega_k(e^*_{ij}) \big)=0.
	\end{align*}
	On the other hand,
	\begin{align*}
		\sum_{ij \in E} \omega_k(e^*_{ij}) \mathbf{e}_l  \times  (f_j-f_i) = -\sum_{i \in V}  \big( \mathbf{e}_l \times f_i  \sum_{ij \in E:i} \omega_k(e^*_{ij}) \big)
		= 0.
	\end{align*}
	By Lemma \ref{lemma:sumzero}, the sums being zero imply the existence of $(u_{kl},Z_{kl})$ as claimed.
	
	Given $(u,Z)\in \field{R}^{V} \times \mathbb{R}^{3F}$ with $\D(u,Z)=(\rho,0)$, the 1-form $\eta$ is closed by Theorem \ref{thm:uZ}. We write $\eta(e_{ij}) = \frac{u_i + u_j}{2} (f_j-f_i) + (f_j-f_i) \times W_{ij}$ where $W_{ij} \perp (f_j-f_i)$. Lemma \ref{lem:hodge} implies that $\eta$ is exact if and only if for every $k=1,2,\ldots,2g$ and $l=1,2,3$,
	\begin{align*}
		0 &= \sum_{ij \in E} \omega(e^*_{ij}) \langle \eta(e_{ij}), \mathbf{e}_l \rangle \\
		&= \sum_{i \in V} \sum_{ij \in E: i} \langle\mathbf{e}_l, \omega_k(e^*_{ij}) (f_j-f_i)\rangle  u_i +  \sum_{ij \in E} \langle\omega_k(e^*_{ij})\mathbf{e}_l \times (f_j-f_i) , W_{ij}\rangle  \\
		&= (\D (u_{kl},Z_{kl}), (u,W)) \\
		&= ((u_{kl},Z_{kl}), \D^*(u,W)) \\
		&= ((u_{kl},Z_{kl}), (\rho,0)) \\
		&= \sum_{i \in V} \rho_i u_{kl,i}.
	\end{align*} 
\end{proof}

\section{Discussion} \label{sec:10}
Proposition \ref{prop:givenu} and Theorem \ref{thm:givenrho} provide two ways to parametrize infinitesimal conformal deformations of a triangulated surface in $\mathbb{R}^3$. For infinitesimally rigid triangulated spheres, one can use scale factors $u$ as intrinsic parameters. For triangulated spheres with $\dim \Ker \D =4$, one can use the change in mean curvature half-density as extrinsic parameters instead. One might be tempting to think that both variables together are sufficient to describe the infinitesimal conformal deformations of a general triangulated surface. However, this is not true.

\begin{definition}\cite{Lam2016}
	A realization $f:V \to \field{R}^3$ of a triangulated surface is called isothermic if there exists $k:E \to \field{R}$ such that for each interior vertex $i$
	\begin{align*}
		\sum_j k_{ij}|f_j - f_i|^2 &=0, \\
		\sum_j k_{ij} (f_j - f_i) &=0.
	\end{align*}
\end{definition}
\begin{proposition}\cite{Lam2016}
	A simply connected surface is isothermic if and only if it admits a non-trivial infinitesimal isometric deformation such that the change in mean curvature half-density vanishes.
\end{proposition}
As a remark, isothermic surfaces are singularities of the space of conformal realizations. They are analogues to infinitesimally flexible surfaces in the case of isometric realizations.
\begin{proposition}\cite{Lam2016}
	The space of infinitesimal conformal deformations of a closed triangulated surface of genus $g$ in $\field{R}^3$ has dimension greater or equal to $ |V|+6-6g$. The inequality is strict if and only if the surface is isothermic.
\end{proposition}

\begin{proposition}\cite{Lam2016}
	The class of isothermic surfaces is M\"{o}bius invariant.
\end{proposition}

Jessen's orthogonal icosahedron in Example \ref{example:jessen} and Bricard's octahedron in Example \ref{example:Bricard} are isothermic surfaces.

For a simply connected surface in the smooth theory, $\dim(\Ker\D)$ is M\"{o}bius invariant \cite[Lemma 26]{Richter1997}. We conjecture that the discrete analogue of this statement holds as well:
\begin{conjecture}
	For a simply connected triangulated surface in space, the nullity of the discrete Dirac operator $\dim(\Ker\D)$ is invariant under M\"{o}bius transformations.
\end{conjecture}

It is interesting to know if we can extend Corollary \ref{cor:givenrho} to triangulated spheres of general combinatorics as like as Gluck's theorem \cite{Gluck1975}. There are two important ingredients to Gluck's theorem. One is Steinitz' theorem that every triangulated sphere admits a strictly convex realization into $\field{R}^3$. The other is Dehn's theorem that all convex realizations are infinitesimally rigid. 

\begin{conjecture} 
	Every abstract triangulated sphere admits a realization in $\mathbb{R}^3$ with $\dim \Ker \D = 4$, which means it does not possess a non-trivial infinitesimal conformal deformation with vanishing change in mean curvature half-density. Hence for almost all triangulated spheres in $\mathbb{R}^3$, their infinitesimal conformal deformations are exactly parametrized by functions $\rho:V \to \field{R}$ with $\sum_i \rho_i=0$ as the change in mean curvature half density.
\end{conjecture}

\bibliographystyle{siam}
\bibliography{conformalspace}   

\begin{thebibliography}{10}

\bibitem{Bobenko2010}
{\sc A.~I. Bobenko, U.~Pinkall, and B.~A. Springborn}, {\em Discrete conformal
  maps and ideal hyperbolic polyhedra}, Geom. Topol., 19 (2015),
  pp.~2155--2215.

\bibitem{Bobenko2005a}
{\sc A.~I. Bobenko and P.~Schr\"{o}der}, {\em Discrete willmore flow}, in
  Proceedings of the Third Eurographics Symposium on Geometry Processing,
  Aire-la-Ville, Switzerland, 2005, Eurographics Association, pp.~101--110.

\bibitem{Whiteley2014}
{\sc R.~Connelly, A.~Ivi\'c~Weiss, and W.~Whiteley}, eds., {\em Rigidity and
  symmetry}, vol.~70 of Fields Institute Communications, Springer, New York,
  2014.

\bibitem{Crane2011}
{\sc K.~Crane, U.~Pinkall, and P.~Schr\"{o}der}, {\em {Spin Transformations of
  Discrete Surfaces}}, ACM Transactions on Graphics, 30 (2011),
  pp.~104:1--104:10.

\bibitem{Dehn1916}
{\sc M.~Dehn}, {\em \"{U}ber die {S}tarrheit konvexer {P}olyeder}, Math. Ann.,
  77 (1916), pp.~466--473.

\bibitem{Desbrun2006a}
{\sc M.~Desbrun, E.~Kanso, and Y.~Tong}, {\em Discrete differential forms for
  computational modeling}, in Discrete differential geometry, vol.~38 of
  Oberwolfach Semin., Birkh\"auser, Basel, 2008, pp.~287--324.

\bibitem{Glickenstein2011}
{\sc D.~Glickenstein}, {\em Discrete conformal variations and scalar curvature
  on piecewise flat two- and three-dimensional manifolds}, J. Differential
  Geom., 87 (2011), pp.~201--237.

\bibitem{Gluck1975}
{\sc H.~Gluck}, {\em Almost all simply connected closed surfaces are rigid}, in
  Geometric topology ({P}roc. {C}onf., {P}ark {C}ity, {U}tah, 1974), Springer,
  Berlin, 1975, pp.~225--239. Lecture Notes in Math., Vol. 438.

\bibitem{Gu2003}
{\sc X.~Gu and S.-T. Yau}, {\em Global conformal surface parameterization}, in
  Proceedings of the 2003 Eurographics/ACM SIGGRAPH Symposium on Geometry
  Processing, Aire-la-Ville, Switzerland, 2003, Eurographics Association,
  pp.~127--137.

\bibitem{Jessen1967}
{\sc B.~Jessen}, {\em Orthogonal icosahedra}, Nordisk Mat. Tidskr, 15 (1967),
  pp.~90--96.

\bibitem{Kamberov1998}
{\sc G.~Kamberov, F.~Pedit, and U.~Pinkall}, {\em Bonnet pairs and isothermic
  surfaces}, Duke Math. J., 92 (1998), pp.~637--644.

\bibitem{Kharevych:2006}
{\sc L.~Kharevych, B.~Springborn, and P.~Schr\"{o}der}, {\em Discrete conformal
  mappings via circle patterns}, ACM Trans. Graph., 25 (2006), pp.~412--438.

\bibitem{Lam2016b}
{\sc W.~Y. Lam}, {\em Discrete minimal surfaces: critical points of the area
  functional from integrable systems}, Int. Math. Res. Not. IMRN,  (2016).

\bibitem{Lam2017}
{\sc W.~Y. Lam}, {\em Minimal surfaces from infinitesimal deformations of
  circle packings},  (2017).
\newblock arXiv:1712.08564.

\bibitem{Lam2016a}
{\sc W.~Y. Lam and U.~Pinkall}, {\em Holomorphic vector fields and quadratic
  differentials on planar triangular meshes}, in Advances in Discrete
  Differential Geometry, A.~I. Bobenko, ed., Springer Berlin Heidelberg, 2016,
  pp.~241--265.

\bibitem{Lam2016}
{\sc W.~Y. Lam and U.~Pinkall}, {\em Isothermic triangulated surfaces}, Math.
  Ann.,  (2016).

\bibitem{Luo2004}
{\sc F.~Luo}, {\em Combinatorial {Y}amabe flow on surfaces}, Commun. Contemp.
  Math., 6 (2004), pp.~765--780.

\bibitem{Pak2010}
{\sc I.~Pak}, {\em {Lectures on Discrete and Polyhedral Geometry}}, 2010.

\bibitem{Pedit1998}
{\sc F.~Pedit and U.~Pinkall}, {\em Quaternionic analysis on {R}iemann surfaces
  and differential geometry}, in Proceedings of the {I}nternational {C}ongress
  of {M}athematicians, {V}ol. {II} ({B}erlin, 1998), no.~Extra Vol. II, 1998,
  pp.~389--400 (electronic).

\bibitem{Richter1997}
{\sc J.~Richter}, {\em {Conformal Maps of a Riemann Surface into Space of
  Quaternions}}, PhD thesis, TU Berlin, 1997.

\bibitem{Williams1984}
{\sc M.~Ro{\v{c}}ek and R.~M. Williams}, {\em The quantization of {R}egge
  calculus}, Z. Phys. C, 21 (1984), pp.~371--381.

\bibitem{Rodin1987}
{\sc B.~Rodin and D.~Sullivan}, {\em The convergence of circle packings to the
  {R}iemann mapping}, J. Differential Geom., 26 (1987), pp.~349--360.

\bibitem{Schlenker2007}
{\sc J.-M. Schlenker}, {\em Small deformations of polygons and polyhedra},
  Trans. Amer. Math. Soc., 359 (2007), pp.~2155--2189.

\bibitem{Schramm1997}
{\sc O.~Schramm}, {\em Circle patterns with the combinatorics of the square
  grid}, Duke Math. J., 86 (1997), pp.~347--389.

\bibitem{Springborn2008}
{\sc B.~Springborn, P.~Schr\"{o}der, and U.~Pinkall}, {\em {Conformal
  Equivalence of Triangle Meshes}}, ACM Transactions on Graphics, 27 (2008).

\bibitem{Stephenson2005}
{\sc K.~Stephenson}, {\em Introduction to circle packing}, Cambridge University
  Press, Cambridge, 2005.

\bibitem{Thurston1982}
{\sc W.~P. Thurston}, {\em The Geometry and Topology of 3-manifolds}, Princeton
  University Notes, Princeton, N.J., 1982.

\end{thebibliography}


\end{document}